\documentclass[11pt,reqno]{amsart}
\usepackage[margin=1in]{geometry}
\usepackage{amsmath,amssymb,amsthm,mathtools}
\usepackage{enumitem}
\usepackage[colorlinks=true,linkcolor=blue,citecolor=blue,urlcolor=blue]{hyperref}

\theoremstyle{plain}
\newtheorem{theorem}{Theorem}[section]

\newtheorem{lemma}[theorem]{Lemma}
\newtheorem{corollary}[theorem]{Corollary}
\theoremstyle{definition}

\newtheorem{remark}[theorem]{Remark}
\newtheorem{conjecture}[theorem]{Conjecture}

\DeclareMathOperator{\rank}{rank}
\DeclareMathOperator{\diag}{diag}
\DeclareMathOperator{\Diag}{Diag}
\DeclareMathOperator{\range}{range}
\DeclareMathOperator{\spn}{span}
\DeclareMathOperator{\Sym}{Sym}
\newcommand{\R}{\mathbb{R}}

\newcommand{\Ell}{\mathcal{E}}
\newcommand{\onevec}{\mathbf{1}}
\newcommand{\SDP}{\mathrm{sdp}}
\newcommand{\SDPproblem}{\mathrm{SDP}}
\newcommand{\rbp}{r_{\mathrm{BP}}}

\begin{document}

\title[Optimal solution ranks in the Max-cut SDP]
      {Optimal solution ranks in the Max-Cut semidefinite relaxation}

\author{Avinash Bhardwaj}
\address{Department of Industrial Engineering and Operations Research\\
Indian Institute of Technology Bombay\\Mumbai, India 400076}
\email{abhardwaj@iitb.ac.in}
\author{Vishnu Narayanan}
\address{Department of Industrial Engineering and Operations Research\\
Indian Institute of Technology Bombay\\Mumbai, India 400076}
\email{vishnu@iitb.ac.in}

\begin{abstract}
For every $n$ and $r$ with $1\le r$ and $\binom{r+1}{2}\le n$, we construct a connected unweighted graph on $n$ vertices whose Max-Cut semidefinite relaxation has a unique optimal solution of rank $r$. The main examples are the Kneser graphs $K(m,2)$, whose unique optima have rank $m-1$. A six-vertex wheel supplies rank three, and attaching pendant vertices preserves both rank and uniqueness. For dense random graphs, we also prove a polynomial lower bound on the optimal rank.
\end{abstract}


\maketitle
\enlargethispage{2pt}

\section{Introduction}\label{sec:intro}

The semidefinite relaxation of Max-Cut replaces the two signs defining a cut by unit vectors. This raises a natural question: how many dimensions can a graph force an optimal representation to use? We show that connected unweighted graphs can require every dimension allowed by the general rank bound, even with a unique optimal Gram matrix.

Let $G$ be a graph on the vertex set $[n]$ with symmetric nonnegative weight matrix $A=(a_{ij})$, zero diagonal, and Laplacian $L=\Diag(A\onevec)-A$. We denote the Goemans--Williamson relaxation \cite{GW95} by $\SDPproblem(G)$ and its optimal value by $\SDP(G)$:

\[
\begin{aligned}
  \SDP(G)&=\max\Bigl\{\tfrac14\langle L,X\rangle:X\in\Ell_n\Bigr\},\\
  \Ell_n&=\{X\in\R^{n\times n}:X\succeq0,\ \diag X=\onevec\},
\end{aligned}
\]
where $\Ell_n$ is the elliptope. A point of $\Ell_n$ is the Gram matrix of a family of unit vectors,
and its rank is the dimension of the space they span. Rank one is the integral case: $X\in\Ell_n$
has rank one exactly when $X=xx^\top$ with $x\in\{\pm1\}^n$, a cut. Let $F(G)$ be the set of optimal
solutions of $\SDPproblem(G)$, which is a face of $\Ell_n$, and define the \emph{optimal rank}
\[
  \rho(G)=\min\{\rank X:X\in F(G)\}.
\]
Thus $\rho(G)$ measures the dimension needed by the best optimal representation, and $\rho(G)=1$ exactly when the relaxation is exact. Bhardwaj~\cite{Bha26} shows that recognising exactness is NP-hard even for unweighted graphs. Optimal ranks have been studied from several angles \cite{HLW21,MW24}, and low-rank formulations also underlie heuristics for Max-Cut \cite{BMZ02}. Here we ask how large $\rho(G)$ can be as a function of the number of vertices.

An upper bound on $\rho(G)$ comes from general theory. $F(G)$ is a nonempty compact face of $\Ell_n$,
so it has an extreme point, which is an extreme point of $\Ell_n$; and by the theorem of Barvinok
\cite{Bar95} and Pataki \cite{Pat98}, an extreme point of a spectrahedron defined by $n$ linear
equations has rank $r$ with $\binom{r+1}2\le n$. Thus an optimal solution of $\SDPproblem(G)$ attains a rank satisfying this bound, and hence
\begin{equation}\label{eq:bp}
  \rho(G)\le\rbp(n):=\max\Bigl\{r:\tbinom{r+1}2\le n\Bigr\}
  =\Bigl\lfloor\tfrac{\sqrt{8n+1}-1}2\Bigr\rfloor .
\end{equation}
The question we answer is whether equality in \eqref{eq:bp} is attained for every $n$, even by connected unweighted graphs.

The distinction between arbitrary objectives and graph objectives is essential here. Extreme points of $\Ell_n$ exist at every rank permitted by \eqref{eq:bp} \cite{GPW90,LT94}. Since every face of a spectrahedron is exposed \cite{RG95}, each is the unique optimum of a suitable linear functional. The same sharp rank bound appears in the extreme Gram dimension of Eisenberg-Nagy, Laurent and Varvitsiotis \cite{ELV14}, where the objective coefficients on a fixed graph may be arbitrary real numbers. For an unweighted graph, the off-diagonal entries of its Laplacian belong to $\{0,-1\}$. Since diagonal terms contribute a constant on $\Ell_n$, it is this restriction on the off-diagonal coefficients that distinguishes our objectives. Our result realizes every admissible rank under this restriction, while also requiring connectivity and uniqueness.

The main construction uses Kneser graphs, familiar from graph coloring and Lov\'asz's work~\cite{Lov78}. The graph $K(m,2)$ has the two-element subsets of $[m]$ as vertices, with two vertices adjacent exactly when the corresponding subsets are disjoint.

\begin{theorem}\label{thm:kneser}
Let $m\ge5$ and $n=\binom m2$. Then the relaxation $\SDPproblem(K(m,2))$ has a unique optimal
solution $X^\ast$, given by
\[
  X^\ast_{ef}=\frac{m\,|e\cap f|-4}{2(m-2)}\qquad(e,f\in\tbinom{[m]}2),
\]
and $\rank X^\ast=m-1=\rbp(n)$. In particular $\rho(K(m,2))=\rbp(n)$.
\end{theorem}

Explicitly, $X^\ast$ is the Gram matrix obtained by assigning to the pair $\{a,b\}$ the unit vector proportional to the centred indicator vector $e_a+e_b-\frac2m\onevec$, where $e_a$ is the $a$-th standard basis vector. Disjoint pairs have inner product $-\frac2{m-2}$ and distinct overlapping pairs have inner product $\frac{m-4}{2(m-2)}$. For $m=5$ this is the Petersen graph, whose unique optimal solution has rank $4=\rbp(10)$.

The Kneser family skips rank three, because $K(4,2)$ is a perfect matching. The wheel formed by joining a hub to all five vertices of a cycle supplies the missing rank, already on six vertices (Theorem~\ref{thm:rank3}). To reach any larger order, we attach pendant vertices: each new leaf receives the negative of its neighbour's vector, preserving both the rank and the uniqueness of the optimum. These observations give a complete realization result.

\begin{theorem}\label{thm:every}
For every $n\ge1$ and every integer $1\le r\le\rbp(n)$, there is a connected unweighted graph on $n$ vertices whose Max-Cut semidefinite relaxation has a unique optimal solution of rank $r$.
\end{theorem}

\begin{corollary}\label{cor:extremal}
For every $n\ge1$, the maximum of $\rho(G)$ over connected unweighted graphs on $n$ vertices equals $\rbp(n)$. The same maximum is obtained over all unweighted graphs, or over graphs with nonnegative real edge weights.
\end{corollary}

Thus every rank permitted by the general theory occurs as the rank of a unique optimum, even when the graph is required to be connected and all its edge weights are one.

The constructions also suggest a question about random graphs: must the dimension grow when the edges are chosen independently? For $G\sim G(n,\frac12)$, where each pair is an edge independently with probability $1/2$, we prove
\[
 \rho(G)\ge c\left(\frac{n}{(\log n)^3}\right)^{1/7}
\]
with probability tending to one, for an absolute constant $c>0$ (Theorem~\ref{thm:random}). Thus, with probability tending to one, every optimal representation of an Erd\H{o}s--R\'enyi graph with edge probability $1/2$ requires polynomially many dimensions. The proof compares the best objective value at bounded rank with a feasible matrix concentrated near the upper spectral edge. We also formulate a conjecture on the growth rate, motivated by the Gaussian spin-glass literature.

The proofs of the realization results follow a common pattern: first locate the range of every optimum, then use the diagonal equations to determine the Gram matrix. For the Kneser graphs, the classical Laplacian eigenvalue bound \cite{MP90,DP93mp} confines this range to the top Laplacian eigenspace. Van Dam and Sotirov \cite{DS16} show that the spectral bound is attained for all walk-regular graphs, a class containing the Kneser graphs. Attaining the value alone does not determine the minimum optimal rank. We prove that the diagonal equations are independent on the relevant space of symmetric matrices, so the entire optimal face consists of one point. This is an application of the classical extremality criterion for correlation matrices \cite{GPW90,LT94}.

The centred-incidence representation is also central to the work of Godsil, Roberson, Rooney, \v{S}\'amal and Varvitsiotis \cite{GRRSV17}, who prove universal completability and unique vector colorability of Kneser graphs. Those problems prescribe or bound individual edge correlations; Max-Cut optimizes their sum. The difference is substantive: $K_4$ has a unique optimal vector coloring, represented by a regular tetrahedron, while its Max-Cut relaxation has both this rank-three optimum and rank-one optima coming from balanced cuts. For $K(m,2)$, we establish uniqueness using the diagonal equations alone within the least adjacency eigenspace.

The wheel uses the same principle with the kernel of a dual slack matrix in place of the top Laplacian eigenspace. Finally, attaching a leaf preserves the optimal rank and uniqueness. This extension is a special case of the vertex-sum constructions in \cite{HLW21}. We collect the necessary facts in Section~\ref{sec:prelim}, prove the Kneser and wheel results in Sections~\ref{sec:kneser} and~\ref{sec:rank3}, and assemble the construction proving Theorem~\ref{thm:every} in Section~\ref{sec:proofs}. The final section gives a polynomial lower bound for random graphs and proposes a conjecture on the growth of their optimal rank.

\section{The optimal face}\label{sec:prelim}

We begin with the geometry shared by the two constructions. The dual of the relaxation is $\min\{\sum_iy_i:\Diag(y)-\frac14L\succeq0\}$. The identity matrix is strictly feasible for the primal problem, so strong duality holds and the dual optimum is attained. If $S$ is an optimal dual slack, complementary slackness confines every optimum $X$ to $\range X\subseteq\ker S$. A scalar diagonal choice gives the familiar eigenvalue bound.

\begin{lemma}[Eigenvalue bound \cite{MP90,DP93mp}]\label{lem:eig}
Let $\lambda=\lambda_{\max}(L)$ and let $E$ be its eigenspace. For every
$X\in\Ell_n$ we have $\langle L,X\rangle\le n\lambda$, with equality if and only if
$\range X\subseteq E$. If some $X\in\Ell_n$ satisfies $\range X\subseteq E$, then
$\SDP(G)=\frac n4\lambda$ and
\[
  F(G)=\{X\in\Ell_n:\range X\subseteq E\}.
\]
\end{lemma}

\begin{proof}
$\lambda I-L\succeq0$ and $X\succeq0$ give $\langle\lambda I-L,X\rangle\ge0$, that is
$\langle L,X\rangle\le\lambda\operatorname{tr}X=n\lambda$. Equality holds if and only if
$\langle\lambda I-L,X\rangle=0$, which for two positive semidefinite matrices means
$(\lambda I-L)X=0$, that is $\range X\subseteq E$.
\end{proof}

\begin{lemma}[Parametrising the face]\label{lem:param}
Let $E\subseteq\R^n$ be a subspace of dimension $g$, let $U\in\R^{n\times g}$ have columns forming a
basis of $E$, and let $u_i^\top$ be the $i$-th row of $U$. Then the map $M\mapsto UMU^\top$ is a
rank-preserving bijection from
\[
  \mathcal M=\{M\in\Sym(\R^g):M\succeq0,\ u_i^\top Mu_i=1\ (i\in[n])\}
\]
onto $\{X\in\Ell_n:\range X\subseteq E\}$.
\end{lemma}

\begin{proof}
A symmetric $X$ has $\range X\subseteq\range U$ if and only if $X=UMU^\top$ for a symmetric $M$,
which is unique because $U$ has full column rank; for the same reason $X\succeq0$ if and only if
$M\succeq0$, and $\rank X=\rank M$. The condition $X_{ii}=1$ reads $u_i^\top Mu_i=1$.
\end{proof}

In these coordinates, uniqueness becomes a question about linear equations: do the conditions $u_i^\top Mu_i=1$ determine $M$? The following form of the correlation-matrix extremality criterion \cite{GPW90,LT94} will be useful.

\begin{lemma}[Spanning forces a single point]\label{lem:span}
In the setting of Lemma~\ref{lem:param}, suppose that the rank-one matrices $u_iu_i^\top$,
$i\in[n]$, span $\Sym(\R^g)$, and that $\mathcal M$ contains a positive definite matrix $M_0$.
Then $\mathcal M=\{M_0\}$. If moreover $E$ is the top eigenspace of $L$, then $F(G)=\{UM_0U^\top\}$
and $\rho(G)=g$.
\end{lemma}

\begin{proof}
The equations $\langle u_iu_i^\top,M\rangle=1$ are linear in $M$, and if the coefficient matrices
span $\Sym(\R^g)$ they have at most one solution. The second statement follows from
Lemmas~\ref{lem:eig} and~\ref{lem:param}.
\end{proof}

The converse holds whenever $M_0\succ0$: a nonzero symmetric matrix orthogonal to every $u_iu_i^\top$ gives sufficiently small feasible perturbations of both signs. Thus uniqueness is equivalent to the absence of such a perturbation. In particular, $n\ge\binom{g+1}{2}$ is necessary. Our constructions achieve equality before any pendant vertices are added.

We will use two small graphs to start the construction. The graph $K_1$ has the unique feasible matrix $(1)$, of rank one. For the triangle $K_3$, the objective of a Gram matrix of unit vectors $v_1,v_2,v_3$ is
\[
  \frac14\left(9-\|v_1+v_2+v_3\|^2\right).
\]
Three unit vectors summing to zero attain the maximum. Taking inner products with each vector shows that all three off-diagonal entries of their Gram matrix are $-1/2$. Hence the optimum for $K_3$ is unique and has rank two.

The following extension lets us add vertices without changing the optimal rank.

\begin{lemma}\label{lem:pendant}
Let $G'$ be obtained from a graph $G$ by adding a new vertex $\ell$ and a single edge $v\ell$ of weight one, where $v\in V(G)$. Then
\[
  \SDP(G')=\SDP(G)+1,\qquad \rho(G')=\rho(G).
\]
If the optimal solution for $G$ is unique, then so is the optimal solution for $G'$.
\end{lemma}

\begin{proof}
For any unit-vector representation of $G'$, the old edges contribute at most $\SDP(G)$ and the new edge contributes $\frac14\|v_\ell-v_v\|^2\le1$. Conversely, an optimal representation of $G$ extends by setting $v_\ell=-v_v$, attaining $\SDP(G)+1$ without increasing the rank. Every optimum for $G'$ must therefore restrict to an optimum for $G$, which proves the equality of optimal ranks. Moreover, equality on the new edge forces $v_\ell=-v_v$. All correlations involving the new vertex are consequently fixed by the old Gram matrix, proving the assertion about uniqueness. 
\end{proof}

\section{The Kneser graphs}\label{sec:kneser}

We now construct the optimal representation and show that it is forced by the diagonal equations. Fix $m\ge5$, let $G=K(m,2)$ and $n=\binom m2$. We write vertices as pairs $e=\{a,b\}$, and $A$ for the adjacency matrix. The graph is regular of degree $d=\binom{m-2}2$, so $L=dI-A$ and the top Laplacian eigenspace is the least adjacency eigenspace.

Let $N\in\R^{n\times m}$ be the incidence matrix of pairs versus points, $N_{e,a}=1$ if $a\in e$.
Then $N^\top N=(m-2)I+J$, where $J$ is the all-ones matrix, because a point lies in $m-1$ pairs and
two distinct points lie in exactly one. Put
\[
  E=\{Nf:f\in\R^m,\ \onevec^\top f=0\}.
\]
Since $N^\top N=(m-2)I$ on $\onevec^\perp$, the map $f\mapsto Nf$ is injective there and $\dim E=m-1$.

\begin{lemma}\label{lem:kneserspec}
The Laplacian $L$ of $K(m,2)$ has eigenvalue $0$ on $\spn(\onevec)$, eigenvalue $d+m-3$ on $E$, and
eigenvalue $d-1$ on the orthogonal complement of $\spn(\onevec)\oplus E$. In particular
$\lambda_{\max}(L)=d+m-3$ and its eigenspace is exactly $E$.
\end{lemma}

\begin{proof}
For $x=Nf$ with $\onevec^\top f=0$ we have $x_{ab}=f_a+f_b$, and
\[
\begin{aligned}
  (Ax)_{ab}&=\sum_{\{c,c'\}\cap\{a,b\}=\varnothing}(f_c+f_{c'})\\
  &=(m-3)\sum_{c\notin\{a,b\}}f_c=-(m-3)(f_a+f_b),
\end{aligned}
\]
because each $c\notin\{a,b\}$ lies in $m-3$ pairs disjoint from $\{a,b\}$. So $Ax=-(m-3)x$ on $E$.
The space $\spn(\onevec)\oplus E$ is the range of $N$, since $\onevec=N(\frac12\onevec)$; its
orthogonal complement is $\ker N^\top=\{x:\sum_{b\ne a}x_{ab}=0\text{ for all }a\}$. For such $x$
also $\sum_ex_e=\frac12\sum_a\sum_{b\ne a}x_{ab}=0$, so
\[
  (Ax)_{ab}=\sum_ex_e-\sum_{b'\ne a}x_{ab'}-\sum_{a'\ne b}x_{a'b}+x_{ab}=x_{ab},
\]
the last term correcting for the pair $\{a,b\}$, which is subtracted twice. Finally $A\onevec=d\onevec$.
Hence $L=dI-A$ has the stated eigenvalues, and $d+m-3>d-1$.
\end{proof}

\begin{proof}[Proof of Theorem~\ref{thm:kneser}]
Let $\pi=I-\frac1mJ$ be the orthogonal projector of $\R^m$ onto $\onevec^\perp$. By
Lemma~\ref{lem:kneserspec} and the identity $N^\top N=(m-2)I$ on $\onevec^\perp$, the orthogonal
projector onto $E$ is $P_E=\frac1{m-2}N\pi N^\top$, with entries
\[
  (P_E)_{ef}=\frac{1}{m-2}\bigl(\onevec_e-\tfrac2m\onevec\bigr)^{\!\top}
  \bigl(\onevec_f-\tfrac2m\onevec\bigr)=\frac{|e\cap f|-\frac4m}{m-2},
\]
where $\onevec_e=e_a+e_b$ is the indicator vector of $e=\{a,b\}$. The matrix
$X^\ast=\frac{n}{m-1}P_E=\frac m{2}P_E$ has entries $\frac{m|e\cap f|-4}{2(m-2)}$; its diagonal
entries ($|e\cap f|=2$) are $1$, it is positive semidefinite, and its range is $E$. By
Lemma~\ref{lem:eig}, $X^\ast$ is optimal and $F(G)=\{X\in\Ell_n:\range X\subseteq E\}$.

It remains to show that every optimum equals $X^\ast$. By Lemma~\ref{lem:span}, it suffices to prove that the matrices $(\pi\onevec_e)(\pi\onevec_e)^\top$ span the symmetric operators on $\onevec^\perp$. Indeed, the rows of $N$ restricted to this subspace are represented by $\pi\onevec_e$, and the positive definite scalar operator $M_0=\frac{m}{2(m-2)}I$ gives $X^\ast$.

Suppose a symmetric operator $M$ on $\onevec^\perp$ is orthogonal to all these rank-one matrices. Extend it to a symmetric matrix $\widetilde M$ on $\R^m$ by setting $\widetilde M\onevec=0$. Orthogonality then reads
\[
  \widetilde M_{aa}+2\widetilde M_{ab}+\widetilde M_{bb}=0\qquad(a\ne b).
\]

Let $\delta_a=\widetilde M_{aa}$ and $T=\sum_a\delta_a$. Summing over $b\ne a$ and using
$\sum_{b\ne a}\widetilde M_{ab}=-\delta_a$ gives $(m-1)\delta_a-2\delta_a+(T-\delta_a)=0$, that is
$(m-4)\delta_a+T=0$ for every $a$. Summing over $a$ gives $(2m-4)T=0$, so $T=0$, and then
$\delta_a=0$ for every $a$ because $m\ne4$. The displayed relation now gives $\widetilde M_{ab}=0$
for $a\ne b$, so $M=0$. By Lemma~\ref{lem:span}, $F(G)=\{X^\ast\}$ and $\rho(G)=m-1$.

Finally $8n+1=4m^2-4m+1=(2m-1)^2$, so $\rbp(n)=m-1$ by \eqref{eq:bp}.
\end{proof}

\begin{remark}\label{rem:square}
An alternative proof of the required nonsingularity follows from the entrywise square of $X^\ast$. The Kneser construction uses exactly $\binom{m}{2}=\dim\Sym(\onevec^\perp)$ vectors, so the diagonal system is square. Indeed, $(X^\ast_{ef})^2=\langle u_eu_e^\top,u_fu_f^\top\rangle$ for its unit Gram vectors $u_e$. Thus $X^\ast\circ X^\ast$ is the Gram matrix of the coefficient matrices in the diagonal equations. With $B=J-I-A$, the adjacency matrix of the triangular graph, it equals $I+\frac{(m-4)^2}{4(m-2)^2}B+\frac{4}{(m-2)^2}A$. On the three subspaces in Lemma~\ref{lem:kneserspec}, its eigenvalues are, respectively,
\[
  \frac m2,\qquad \frac{m(m-4)^2}{4(m-2)^2},\qquad \frac{m^2}{2(m-2)^2}.
\]
All are positive for $m\ge5$. At $m=4$ the middle eigenvalue vanishes: $K(4,2)$ is a perfect matching and has $\rho=1$, explaining the missing rank in the Kneser family.
\end{remark}

\begin{remark}\label{rem:value}
The optimal value is $\SDP(K(m,2))=\frac n4(d+m-3)=\frac{m^2(m-1)(m-3)}{16}$.
\end{remark}

\section{Rank three}\label{sec:rank3}

The Kneser family leaves one rank to account for. The triangle supplies rank two, but $K(4,2)$ is a perfect matching and has an integral optimum. For rank three we turn to a familiar graph: the wheel consisting of a five-cycle and a hub adjacent to every cycle vertex. It has six vertices, the smallest order allowed by the Barvinok--Pataki bound for rank three.

\begin{theorem}\label{thm:rank3}
Let $H$ be the unweighted graph obtained by joining a new vertex to every vertex of $C_5$. Then the Max-Cut semidefinite relaxation of $H$ has a unique optimal solution of rank three and optimal value $(23+4\varphi)/4$, where $\varphi=(1+\sqrt5)/2$.
\end{theorem}

\begin{proof}
Write $h$ for the hub, label the cycle vertices $0,\ldots,4$ in cyclic order, and put
\[
  \varphi=\frac{1+\sqrt5}{2},\quad b=\varphi+2,\quad a=b^{-1},\quad s=\sqrt{1-a^2}.
\]
Assign the hub the unit vector $u_h=(0,0,1)$ and assign the cycle vertices
\[
  u_j=(s\cos\theta_j,s\sin\theta_j,-a),\quad \theta_j=\frac{4\pi j}{5}\quad(0\le j\le4).
\]
The five cycle vectors lie on a circle below the hub vector, with consecutive cycle vertices separated by an angle of $4\pi/5$ around that circle. To see the choice of height, replace $a$ by $t\in[0,1]$ in this representation. The sum of the edge inner products is $\frac52((\varphi+2)t^2-2t-\varphi)$, minimized at $t=a$. We next certify that this representation is optimal among all feasible Gram matrices. Let $U\in\R^{6\times3}$ have the displayed rows. Its columns are independent, so $X^*=UU^\top$ is feasible and has rank three.

Let $A_C$ be the adjacency matrix of $C_5$ and consider, in the vertex order $h,0,\ldots,4$, the matrix
\[
  S=\begin{pmatrix}
    5/b&\onevec^\top\\
    \onevec&A_C+\varphi I
  \end{pmatrix}.
\]
The eigenvalues of $A_C$ are $2\cos(2\pi k/5)$, $0\le k\le4$. On the nonconstant cycle eigenspaces, $S$ therefore has eigenvalues $0$ and $\sqrt5$, each twice. On the remaining subspace, with the constant cycle vector normalized to unit length, it is represented by
\[
  \begin{pmatrix}5/b&\sqrt5\\\sqrt5&b\end{pmatrix}.
\]
This matrix has determinant zero and trace $5/b+b=5$. Thus $S\succeq0$ and $\dim\ker S=3$.

The identity $2\cos(4\pi/5)=-\varphi$ shows that $S$ annihilates the first two columns of $U$. For the third column, the hub equation is $5/b-5a=0$ and each cycle equation is $1-(2+\varphi)a=0$. Hence $SU=0$ and $\ker S=\range U$.

Let $A_H$ be the adjacency matrix of $H$. Since $S$ and $A_H$ have the same off-diagonal entries, every $X\in\Ell_6$ satisfies
\[
  \langle A_H,X\rangle=\langle S,X\rangle-5/b-5\varphi\ge-5/b-5\varphi.
\]
Equality holds for $X^*$ because $SU=0$. Minimizing $\langle A_H,X\rangle$ is equivalent to maximizing the Max-Cut objective, so $X^*$ is optimal. For any optimum $X$, equality gives $SX=0$, and therefore $\range X\subseteq\ker S=\range U$.

The graph has ten edges, so the optimal value is $5+(5/b+5\varphi)/4=(23+4\varphi)/4\approx7.368$. Its maximum cut has value $7$: if $k$ cycle vertices lie opposite the hub, at most $k+2\min\{k,5-k\}\le7$ edges cross the cut, and equality is attained by placing cycle vertices $0,2,4$ opposite the hub.

It remains to show that the diagonal equations allow no other optimum. By Lemma~\ref{lem:param}, write $X=UMU^\top$ with $M\in\Sym(\R^3)$ positive semidefinite. The hub equation gives $M_{33}=1$. The equation $u_j^\top Mu_j=1$ at each cycle vertex becomes
\[
\begin{split}
  0={}&\frac{s^2}{2}(M_{11}+M_{22}-2)\\
     &+\frac{s^2}{2}(M_{11}-M_{22})\cos(2\theta_j)
      +s^2M_{12}\sin(2\theta_j)\\
     &-2asM_{13}\cos\theta_j-2asM_{23}\sin\theta_j.
\end{split}
\]
The five functions $1,\cos\theta,\sin\theta,\cos2\theta,\sin2\theta$ are linearly independent when evaluated at the five angles $\theta_j$. Indeed, multiplying a vanishing linear combination by $e^{2i\theta}$ gives a polynomial of degree at most four with five distinct roots $e^{i\theta_j}$. All five coefficients in the displayed equation must therefore vanish. Since $a,s>0$, this gives $M_{12}=M_{13}=M_{23}=0$ and $M_{11}=M_{22}=1$. Hence $M=I$, proving uniqueness.
\end{proof}

\begin{remark}\label{rem:wheel-slack}
The wheel illustrates why it is useful to allow a general dual slack matrix. Its largest Laplacian eigenvalue is $6$ and is simple: an eigenvector has entry $-5$ at the hub and entry $1$ at every cycle vertex. No Gram matrix with unit diagonal can have its range in this eigenspace. The matrix $S$ in the proof instead uses a diagonal shift of the adjacency matrix tailored to the hub and the cycle vertices.
\end{remark}

\section{Realizing every rank}\label{sec:proofs}

Each rank now has a connected unweighted example on the smallest possible number of vertices. Attaching leaves increases the order one vertex at a time while preserving the unique optimum and its rank.

\begin{proof}[Proof of Theorem~\ref{thm:every}]
Let $1\le r\le\rbp(n)$, so $\binom{r+1}{2}\le n$. For $r=1$ start with $K_1$, and for $r=2$ start with $K_3$; as observed in Section~\ref{sec:prelim}, their optima are unique and have ranks one and two, respectively. For $r=3$ use the six-vertex wheel of Theorem~\ref{thm:rank3}. For $r\ge4$ use $K(r+1,2)$, whose unique optimum has rank $r$ by Theorem~\ref{thm:kneser}.

These starting graphs are connected. For the Kneser graph, disjoint pairs are adjacent, while two distinct intersecting pairs have a common neighbour: since $r+1\ge5$, a two-element set can be chosen outside their union. Each starting graph has $\binom{r+1}{2}$ vertices. Attach $n-\binom{r+1}{2}$ pendant vertices and apply Lemma~\ref{lem:pendant} repeatedly. The resulting graph is connected and unweighted, and its optimum is unique and still has rank $r$.
\end{proof}

\begin{proof}[Proof of Corollary~\ref{cor:extremal}]
The upper bound \eqref{eq:bp} applies to every graph under consideration. Theorem~\ref{thm:every} with $r=\rbp(n)$ attains it in the smallest of the three classes, the connected unweighted graphs.
\end{proof}

Can the bound $\rho(G)=\rbp(n)$ be attained with a unique optimum by a $2$-connected graph at every order $n\ge3$, or by a connected regular graph at every sufficiently large order?

\section{Random graphs}\label{sec:open}

The constructions above attain the largest possible optimal rank. We now ask what happens when the edges are chosen independently. The next result shows that a dense random graph forces every optimal representation to use a polynomially growing number of dimensions.

\begin{theorem}\label{thm:random}
There is an absolute constant $c>0$ such that, for $G\sim G(n,\frac12)$,
\[
 \mathbb P\!\left\{\rho(G)\ge c\left(\frac{n}{(\log n)^3}\right)^{1/7}\right\}\longrightarrow1.
\]
\end{theorem}

The idea is to separate objective values: if every representation in at most $r$ dimensions performs worse than one feasible matrix, then every optimum has rank greater than $r$. After centering and scaling the adjacency matrix, we compare both values with the spectral edge $2$. A Gaussian comparison produces a loss of order $1/r$ at bounded rank, and a spectral construction comes closer to $2$. The main work is to keep the errors small enough as $r$ grows.

\begin{proof}[Proof of Theorem~\ref{thm:random}]
Let $A$ be the adjacency matrix and put
\[
 W=\frac{J-I-2A}{\sqrt n},\qquad
 F_n(X)=\langle W,X\rangle-\frac{\onevec^\top X\onevec}{\sqrt n}.
\]
Centering removes the mean edge density: the off-diagonal entries of $W$ are independent symmetric signs divided by $\sqrt n$, and $W_{ii}=0$. The term involving $\onevec$ in $F_n$ accounts for the removed mean. The identity
\[
 \langle A,X\rangle=\tfrac12\bigl(\onevec^\top X\onevec-n-\sqrt n\langle W,X\rangle\bigr)
\]
shows that maximizing $F_n$ on $\Ell_n$ is equivalent to the Max-Cut relaxation. Define
\[
\begin{aligned}
 M_r(W)&=\max\{\langle W,X\rangle:X\in\Ell_n,\ \rank X\le r\},\\
 c_r&=\frac{\mathbb E\|g\|_2}{\sqrt r}
 =\sqrt{\frac2r}\frac{\Gamma((r+1)/2)}{\Gamma(r/2)},\quad g\sim N(0,I_r).
\end{aligned}
\]
The Gaussian norm ratio satisfies $2c_r=2-1/(2r)+O(r^{-2})$. Our two estimates are, with probability tending to one,
\begin{align}
 \begin{split}
 \frac{M_r(W)}n\le 2c_r+C\biggl[&r\sqrt{\frac{\log n}{n}}\\
 &+\frac{(r\log n)^{3/4}}{n^{1/4}}\biggr],
 \end{split}\label{eq:rankbound}\\
 \frac1n\max_{X\in\Ell_n}F_n(X)&\ge2-C_0n^{-1/5}.\label{eq:fullbound}
\end{align}
The constants are absolute, and \eqref{eq:rankbound} holds for deterministic integer sequences $1\le r\le\sqrt n$. The first estimate bounds all rank-$r$ representations; the second supplies a feasible competitor. We establish them in three steps, then compare the errors with the gap $1/r$.

\smallskip
\noindent\emph{1. The low-rank gap for Gaussian entries.}
First replace $W$ by a matrix $B$ from the normalized Gaussian orthogonal ensemble (GOE), whose independent upper-triangular entries are centered Gaussian with variances $1/n$ off the diagonal and $2/n$ on it. For a matrix $U\in\R^{n\times r}$ with unit rows $u_i^\top$, set $H_U=\langle B,UU^\top\rangle$ and $Q_U=U^\top U/n$. The matrix $Q_U$ records how the row vectors are distributed among the $r$ coordinates. It satisfies $Q_U\succeq0$, $\operatorname{tr}Q_U=1$, and
\[
 \mathbb E H_UH_V=\frac2n\|U^\top V\|_F^2.
\]
The quadratic process $H_U$ is easier to bound when compared with a process that is linear in the rows. Fix $Q\succeq0$ of trace one and independent $g_i\sim N(0,I_r)$, and let $Y_U=2\sum_i g_i^\top Q^{1/2}u_i$. Direct expansion gives
\[
\begin{aligned}
 &\mathbb E(H_U-H_V)^2-\mathbb E(Y_U-Y_V)^2\\
 &\quad=\frac2n\bigl(\|U^\top U-nQ\|_F^2+\|V^\top V-nQ\|_F^2\\
 &\hspace{38mm}-2\|U^\top V-nQ\|_F^2\bigr).
\end{aligned}
\]
Thus the difference is at most $4n\delta^2$ whenever $Q_U$ and $Q_V$ lie within Frobenius distance $\delta$ of $Q$.

To use this comparison, we discretize the row vectors and group configurations with nearby matrices $Q_U$. Use $\delta=n^{-3}$ for both discretizations. A $\delta$-net of the unit sphere with at most $(3/\delta)^r$ points gives a row net $\mathcal N$ with $\log|\mathcal N|\le nr\log(3/\delta)$. Cover the trace-one positive semidefinite matrices by at most $K=(1+2/\delta)^{r(r+1)/2}$ Frobenius balls of radius $\delta$, centered in that set, and assign each $U\in\mathcal N$ to one ball according to $Q_U$.

Within a cell centered at $Q$, add to $Y_U$ independent Gaussian variables $Z_U$ of variance $2n\delta^2$. The increments of $Y_U+Z_U$ dominate those of $H_U$, so the Gaussian comparison inequality \cite[Theorem 7.2.11]{Ver18} gives
\[
\begin{aligned}
 \mathbb E\max_{U\text{ in cell}}H_U
 &\le2n\mathbb E\sqrt{g^\top Qg}+2\delta\sqrt{n\log|\mathcal N|}\\
 &\le2nc_r+C\delta n\sqrt{r\log(3/\delta)}.
\end{aligned}
\]
Indeed, $\mathbb E\sqrt{g^\top Qg}$ is symmetric and concave in the eigenvalues of $Q$, whose sum is one; it is maximized at $Q=I/r$. This is where the dimension-dependent gap enters.

Each cell maximum is $\sqrt{2n}$-Lipschitz in the independent standard Gaussian coordinates generating $B$: the squared coefficient norm of $H_U$ is $(2/n)\|UU^\top\|_F^2\le2n$. Write $Z_k$ for the maximum over the $k$-th nonempty cell. Gaussian concentration \cite[Theorem 5.2.2]{Ver18} and a union bound give, for $t\ge0$,
\[
 \mathbb P\!\left\{\max_k(Z_k-\mathbb EZ_k)>t\right\}
 \le\min\{1,K e^{-c_1t^2/n}\}.
\]
Here $c_1>0$ is absolute, and no independence between the cell maxima is required. Integrating this bound, splitting the integral at $\sqrt{(n/c_1)\log(2K)}$, yields
\[
 \mathbb E\max_k Z_k
 \le\max_k\mathbb EZ_k+C\sqrt{n\log(2K)}.
\]
Here $C$ is absolute. Since $\log(2K)\le Cr^2\log n$, the contribution after division by $n$ is at most $Cr\sqrt{\log n/n}$. For any symmetric $T$, approximating each unit row within $\delta$ changes $\langle T,UU^\top\rangle$ by at most $2n\delta\|T\|_{\rm op}$. Since $\mathbb E\|B\|_{\rm op}=O(1)$, we obtain
\begin{equation}\label{eq:gaussian}
 \frac{\mathbb E M_r(B)}n\le2c_r+C r\sqrt{\frac{\log n}{n}}.
\end{equation}
To compare with the zero-diagonal matrix $W$, put $B^\circ=B-\Diag(B_{11},\ldots,B_{nn})$. Every feasible matrix has unit diagonal, so, pointwise,
\[
 M_r(B)=M_r(B^\circ)+\operatorname{tr}B.
\]
As $\mathbb E\operatorname{tr}B=0$, these two maxima have the same expectation. The same identity holds when the maximum is restricted to the row net. Thus only the off-diagonal entries need to be replaced.

\smallskip
\noindent\emph{2. Replacing Gaussian entries by symmetric signs.}
We now transfer the estimate to the random graph. The maximum is not smooth in the matrix entries, so we first replace it by a logarithm of a sum of exponentials. This allows a Taylor expansion as entries are replaced one at a time, following the approach for random semidefinite programs in \cite{MS}. On the same row net, put
\[
\begin{aligned}
 h_U(\xi)&=\frac2{\sqrt n}\sum_{i<j}\xi_{ij}u_i^\top u_j,\\
 f_\beta(\xi)&=\frac1{n\beta}\log\sum_{U\in\mathcal N}e^{\beta h_U(\xi)}.
\end{aligned}
\]
The smooth maximum differs from $n^{-1}\max_U h_U$ by a quantity between zero and
$\log|\mathcal N|/(n\beta)\le Cr\log n/\beta$. For one coordinate, $|\partial_{ij}h_U|\le2/\sqrt n$. The fourth derivative of $f_\beta$ is $\beta^3/n$ times the difference between the fourth centered moment and three times the squared variance of these slopes under the exponential weights. The bound on the slopes therefore gives
\[
 |\partial_{ij}^4f_\beta|\le C\beta^3/n^3.
\]
For completeness, fix one coordinate and condition on all the others in the current mixture of Gaussian and sign entries. Write $\phi(t)$ for the resulting smooth maximum. Taylor's formula gives
\[
 \phi(t)=\sum_{j=0}^3\frac{\phi^{(j)}(0)}{j!}t^j+R(t),
 \qquad |R(t)|\le\frac{C\beta^3}{n^3}|t|^4.
\]
The derivative bound is uniform in $t$ and in the conditioned entries. A standard Gaussian $g$ and a symmetric sign $\varepsilon$ have the same moments through degree three, while $\mathbb Eg^4=3$ and $\mathbb E\varepsilon^4=1$. Their conditional expectations therefore cancel in the polynomial part, leaving
\[
 |\mathbb E_g\phi(g)-\mathbb E_\varepsilon\phi(\varepsilon)|
 \le C\beta^3/n^3.
\]
Taking expectation over the other entries and telescoping over the $\binom n2$ replacements gives total error at most $C\beta^3/n$, with a constant independent of $r$. Adding the smoothing errors and the $O(n^{-3})$ expected net errors gives
\[
 \left|\frac{\mathbb E M_r(W)-\mathbb E M_r(B)}n\right|
 \le C\left(\frac{r\log n}{\beta}+\frac{\beta^3}{n}+n^{-3}\right).
\]
Here both ensembles have bounded expected operator norm. The choice $\beta=(nr\log n)^{1/4}$ balances the smoothing and replacement errors; applying \eqref{eq:gaussian} now gives the desired expectation bound. To pass to a probability bound, observe that changing one sign changes $M_r(W)/n$ by at most $4n^{-3/2}$. The bounded-differences inequality \cite[Theorem 2.9.1]{Ver18} gives
\[
 \mathbb P\{M_r(W)-\mathbb E M_r(W)>nt\}\le e^{-cnt^2}.
\]
Choosing $t=C\sqrt{\log n/n}$ proves \eqref{eq:rankbound}, with failure probability any prescribed negative power of $n$ after increasing $C$. Matching the third moments is essential to this error estimate and is available here because the graph density is $1/2$.

\smallskip
\noindent\emph{3. A feasible matrix near the spectral bound.}
We seek a positive semidefinite matrix concentrated on eigenvectors whose eigenvalues are close to $2$. Its diagonal must also be close to constant, so that enforcing unit diagonal changes the objective only slightly. The local semicircle law provides this control. Set $d=n^{-1/5}$ and $z=2+d$. Spectral-edge rigidity \cite[Theorem 2.9]{BGK} gives $\lambda_{\max}(W)\le2+d/4$ with probability tending to one. On this event let
\begin{align*}
 P & =(zI-W)^{-4},\quad D=\Diag(P_{11},\ldots,P_{nn}),\\
 X & =D^{-1/2}PD^{-1/2}.
\end{align*}
The matrix $P$ assigns weight $(z-\lambda)^{-4}$ to an eigenvector with eigenvalue $\lambda$, favouring those near the upper spectral edge. Diagonal normalization gives $X\succeq0$ and $\diag X=\onevec$. We will check both its objective value and the mean penalty in $F_n$.

Write $Z_n=O_\prec(a_n)$ when, for every fixed $\eta,D_0>0$, $\mathbb P\{|Z_n|>n^\eta a_n\}\le n^{-D_0}$ for all sufficiently large $n$, including maxima over the indicated indices. Put $q(w)=\int(w-x)^{-1}\,d\mu_{\rm sc}(x)$, where $\mu_{\rm sc}$ is the semicircle law, with density $\sqrt{4-x^2}/(2\pi)$ on $[-2,2]$. We justify its application on the contour explicitly. The matrix $W$ satisfies \cite[Definition 2.2]{BGK}: its upper-triangular entries are independent and centered, the off-diagonal variances are $1/n$, and every moment of $\sqrt n W_{ij}$ is uniformly bounded. The variance condition there is $\mathbb E|W_{ij}|^2=n^{-1}(1+O(\delta_{ij}))$, where $\delta_{ij}$ is the Kronecker delta; it permits zero diagonal variance.

Let $\Gamma=\{w:|w-z|=d/4\}$. On its upper semicircle, $\kappa=\Re w-2\in[3d/4,5d/4]$ and $0<\Im w\le d/4$. Taking, for example, $\tau=1/10$, we have $3d/4\ge n^{-2/3+\tau}$ for large $n$. Thus these points belong to the domain $\{E+i\eta:|E|\ge2+n^{-2/3+\tau},\ \eta>0\}$ of \cite[Theorem 10.3]{BGK}. This theorem imposes no positive lower cutoff on $\eta$, so net points may approach the real endpoints arbitrarily closely. Its entrywise estimate, with $\kappa+\Im w\asymp d$, gives
\[
 [(wI-W)^{-1}]_{ii}=q(w)+O_\prec(n^{-1/2}d^{-1/4})
\]
uniformly in the deterministic upper-semicircle point $w$ and in $i$. The source uses $(W-wI)^{-1}$ and the opposite sign for the semicircle transform; changing both signs gives our convention.

To obtain a simultaneous bound, choose a deterministic $n^{-3}$-net of the upper semicircle using points with positive imaginary part, also approximating its two real endpoints. It has at most $Cn^3$ points. A union bound over the net and the $n$ diagonal entries preserves stochastic domination, since the local law permits any prescribed polynomial failure exponent. On the rigidity event above, the closed disk bounded by $\Gamma$ lies at distance at least $d/2$ from the spectrum of $W$. For $R(w)=(wI-W)^{-1}$ and points $w,v$ in this disk, the resolvent identity gives
\[
 \|R(w)-R(v)\|_{\rm op}\le4d^{-2}|w-v|.
\]
The integral defining $q$ gives the analogous bound $|q(w)-q(v)|\le Cd^{-2}|w-v|$. Interpolation from the net therefore costs at most $Cn^{-3}d^{-2}=Cn^{-13/5}$, smaller than $n^{-1/2}d^{-1/4}=n^{-9/20}$. This includes the real endpoints. Finally, $R(\overline w)=\overline{R(w)}$ and $q(\overline w)=\overline{q(w)}$ extend the same bound to the lower semicircle. The estimate consequently holds simultaneously on all of $\Gamma$.

Both resolvents are analytic in a neighborhood of the closed disk on this event. Applying Cauchy's formula to their difference multiplies its contour bound by at most $C_jd^{-(j-1)}$ for the $(j-1)$-st derivative.

Since the $(j-1)$-st derivative of $(zI-W)^{-1}$ is $(-1)^{j-1}(j-1)!(zI-W)^{-j}$, we obtain, for $j=3,4$,
\[
\begin{aligned}
 [(zI-W)^{-j}]_{ii}&=a_j+O_\prec(n^{-1/2}d^{3/4-j}),\\
 a_j&=\int(z-x)^{-j}\,d\mu_{\rm sc}(x)\asymp d^{3/2-j}.
\end{aligned}
\]
The last estimate follows by substituting $2-x=dt$ and using the square-root density at the edge. These formulas serve two purposes: the case $j=4$ controls the diagonal normalization, while $j=3$ controls the objective through $WP=zP-(zI-W)^{-3}$. Both have the same relative error. In particular,
\[
 u:=\max_i|P_{ii}/a_4-1|
 =O_\prec(n^{-1/2}d^{-3/4})\!=O_\prec(n^{-7/20}).
\]
The diagonal normalization is therefore a small perturbation. Let $C_D=(D/a_4)^{-1/2}$. Since $\|C_D-I\|_{\rm op}=O(u)$ and $P\succeq0$, the trace norm satisfies
\[
\begin{aligned}
 \|X-P/a_4\|_{\rm tr}
 &=\|C_D(P/a_4)C_D-P/a_4\|_{\rm tr}\\
 &\le O(u)\operatorname{tr}(P/a_4)=O(nu).
\end{aligned}
\]
Using $WP=zP-(zI-W)^{-3}$ and $\|W\|_{\rm op}=O(1)$ with high probability, we conclude that
\[
 \frac{\langle W,X\rangle}n
 =z-\frac{a_3}{a_4}+O_\prec(n^{-7/20})
 \ge2-Cd+O_\prec(n^{-7/20}).
\]
Indeed, the estimates $a_3\asymp d^{-3/2}$ and $a_4\asymp d^{-5/2}$ give $0<a_3/a_4\le Cd$; together with $z=2+d$, this proves the last inequality. Moreover, $\|X\|_{\rm op}\le\|P\|_{\rm op}/\min_iP_{ii}\le Cd^{-3/2}$, and therefore
\[
 \frac{\onevec^\top X\onevec}{n^{3/2}}
 \le\frac{\|X\|_{\rm op}}{\sqrt n}\le Cn^{-1/5}.
\]
The stochastic error is $o(n^{-1/5})$ even after a sufficiently small fixed power loss. This proves \eqref{eq:fullbound}.

\smallskip
\noindent\emph{Comparing the two values.}
It remains to make the errors smaller than the $1/r$ loss in $2c_r$. Choose
$r=\lfloor c(n/(\log n)^3)^{1/7}\rfloor$. Then
\[
\begin{gathered}
 r^2\sqrt{\frac{\log n}{n}}=o(1),\qquad rn^{-1/5}=o(1),\\
 r\,n^{-1/4}(r\log n)^{3/4}\le c^{7/4}.
\end{gathered}
\]
Take $c>0$ small enough that the constant in \eqref{eq:rankbound} times $c^{7/4}$ is less than $1/8$. Equations \eqref{eq:rankbound}--\eqref{eq:fullbound} then give, with probability tending to one,
\[
\begin{aligned}
 \max_{\substack{X\in\Ell_n\\\rank X\le r}}F_n(X)
 &\le M_r(W)\le n\left(2-\frac1{4r}\right)\\
 &<\max_{X\in\Ell_n}F_n(X).
\end{aligned}
\]
Every optimizer has rank greater than $r$, proving the theorem.
\end{proof}

The exponent $1/7$ comes from balancing the comparison errors against the $1/r$ gap; it need not describe the actual growth rate.

There is a more specific prediction suggested by the Gaussian infinite-component spin glass. Hastings~\cite{Hastings00} derived a ground-state dimension of order $n^{2/5}$. Aspelmeier and Moore~\cite{AM04} supported this scaling numerically and gave an alternative argument balancing a finite-dimension energy correction of order $1/r$ against a spectral correction of order $(r/n)^{2/3}$. Balancing these two terms gives $r$ of order $n^{2/5}$. Motivated by this prediction, we propose the following statement for unweighted random graphs.

\begin{conjecture}\label{conj:typical}
If $G\sim G(n,\frac12)$, then
\[
 \frac{\log\rho(G)}{\log n}\xrightarrow{\mathbb P}\frac25.
\]
\end{conjecture}

The spin-glass calculations motivate this conjecture but do not establish it for the graph model. In particular, agreement of leading optimal values does not imply agreement of exact optimal ranks. Theorem~\ref{thm:random} supplies a polynomial lower bound; even an upper bound $\rho(G)=o(\sqrt n)$ with probability tending to one remains open here.

\bibliographystyle{amsplain}
\bibliography{kneser}
\end{document}